\documentclass[12pt,leqeq]{article}
\usepackage{amssymb,amsthm}
\usepackage{amsmath}
\usepackage{amscd}
\usepackage{xy}
\xyoption{all}
\usepackage{graphicx}
\usepackage{graphics}
\usepackage{graphpap}
\usepackage{pstricks}
\usepackage{pst-all}
\usepackage{pstcol}
\usepackage{color}
\usepackage{caption}
\newtheorem{thm}{Theorem}
\newtheorem{lem}[thm]{Lemma}
\newtheorem{cor}[thm]{Corollary}
\newtheorem{prop}[thm]{Proposition} 
\newtheorem{rem}[thm]{Remark}

\newtheorem{defn}[thm]{Definition}
\newtheorem{exmp}[thm]{Example}
\newtheorem{iden}[thm]{Identity}

\date{}
\begin{document}
\setlength{\baselineskip}{16pt}
\title{Continuous Analogues for the Binomial Coefficients and the Catalan Numbers}
\author{Leonardo Cano \ \ and \ \  Rafael D\'\i az}
\maketitle

\begin{abstract}
Using techniques from the theories of convex polytopes, lattice paths, and indirect influences on directed manifolds, we construct continuous analogues for the binomial coefficients and the Catalan numbers. Our approach for constructing these analogues can be applied to a wide variety of combinatorial sequences. As an application we develop a continuous analogue for the binomial distribution.
\end{abstract}

\section{Introduction}

In this work we construct continuous analogues for
the binomial coefficients and  the Catalan numbers. Our
constructions are based on the theory of convex polytopes, the theory of
lattice paths,  and the theory of indirect influences on directed manifolds. We
introduce our methodology for finding continuous analogues  --
applicable to many kinds of combinatorial objects -- trough the
following table:

\begin{center}
\begin{tabular}{|l | l|}
  \hline
   Combinatorial Object & Continuous Analogue   \\
  \hline   \hline
Lattice $\ \mathbb{Z}^d \ $   & Smooth manifold $\ \mathbb{R}^d$    \\
   \hline
Lattice step vector  $\ v \in \mathbb{Z}^d \ $   & Constant vector field $v$ on $\ \mathbb{R}^d$    \\
   \hline
   Lattice step vectors $\ v_1,...,v_k   \in   \mathbb{Z}^d \ $ & Directed manifold $\ (\mathbb{R}^d, v_1,...,v_k)$  \\
  \hline
  Lattice paths  & Directed paths   \\

  \hline

   Finite pattern decomposition & Countable pattern decomposition   \\
  \hline
    $\sharp$ Integer points in interior of polytopes & Volume of polytopes   \\
  \hline

  Binomial coefficients $\ {n \choose k}$  & Continuous binomial coefficients $\ {x \brace s}$   \\
  \hline
   Catalan numbers $\ c_n$   & Continuous Calatan numbers $\ C(x)$  \\
  \hline
\end{tabular}
\end{center}

A polytope in $\ \mathbb{R}^d \supseteq \mathbb{Z}^d\ $ gives rise to the weighted poset of its faces (ordered by inclusion,)
with the weight of a face being the number of integer points in its relative interior. Restricting attention to the lowest and highest elements of this poset, a couple of combinatorial problems arise whenever we are given a convex polytope $\ P \subseteq  \mathbb{R}^d: \ $ count the number of vertices of $\ P,\ $ and count the number of integer points in $\ P^{\circ},\ $ the relative interior of $\ P. \  $  Accordingly, a couple of different meanings can be given to the problem of finding a convex polytopal interpretation, or realization,  of a sequence $\ a_n \ $ of natural numbers:
\begin{description}

 \item[I.] Find a sequence $\ P_n \subseteq  \mathbb{R}^{d_n} \  $ of polytopes such that  $\ a_n  =  |\mbox{vertices}(P_n)| .$

 \item[II.] Find a sequence $\ P_n \subseteq  \mathbb{R}^{d_n} \  $ of polytopes such that $\ a_n  =  |P_n^{\circ} \cap \mathbb{Z}^{d_n}| .$
\end{description}

Clearly, in both cases, one can always find a (non-unique) sequence of polytopes with the required property, just as it happens when we consider interpretations of the natural numbers as the cardinality of arbitrary finite sets.  Thus,  we are actually interested in finding nice polytopal interpretations having additional properties. The reader may wonder why we count points in the interior of polytopes, and not in the whole polytope. To a great extent both choices are equally valid, and indeed they are tightly related by the M$\ddot{\mbox{o}}$bius inversion formula, and the Ehrhart reciprocity theorem \cite{E}. We give preponderance to interior integral points because that is what arises in our general constructions in Sections \ref{2} and \ref{3}. \\

For the Catalan numbers $\ c_n  =  \frac{1}{n+1}{2n \choose n} \ $  problem I \ admits a nice answer in terms of the Stasheff's  associahedra, which play a prominent role in the study of algebras associative up to homotopy, and particularly in the construction of the operad for $A_{\infty}$-algebras \cite{mss}. The associahedra were first constructed by Tamari, coming from a different viewpoint, who gave a combinatorial description of the poset of its faces \cite{mps}.\\

Solutions to  problem II \ lead naturally to the construction of continuous analogues for the sequence of natural numbers $\ a_n \ $ as follows: the numbers $\ a_n \ $ count the integral points in the interior of the polytopes $\ P_n \subseteq  \mathbb{R}^{d_n}, \  $ and we can think of the volume $\ \mathrm{vol}(P_n) \ $ as counting -- actually measuring --  points in  $\ P_n \ $ after the integrality restrictions are lifted.  Therefore, one feels entitled to regard the real numbers $\ \mathrm{vol}(P_n) \ $ as being continuous analogues for the natural numbers $\  a_n  =  |P_n^{\circ} \cap \mathbb{Z}^{d_n}|.$  Although a bit vague for the moment, this analogy will become much clearer when applied to the polytopes coming from the theory of lattice paths studied in this work. In this case our analogy simply amounts to replacing  lattice paths by directed paths (i.e. polygonal paths with specified tangent vectors), a process that can be intuitively grasped by comparing Figures 5 and 6. For more on the theory of lattice paths the reader may consult
Banderier and Flajolet \cite{fla}, Humphreys \cite{hum},  Krattenthaler \cite{kra}, Mohanty \cite{moha}, and Narayana \cite{nara}.\\

The relation between the numbers $ \ |P^{\circ} \cap \mathbb{Z}^{d}| \ $  and $\ \mathrm{vol}(P) \ $  for a convex polytope $\ P \ $ is much deeper than what one might naively think. Let us highlight a few points that   show the depth of  this relationship:
\begin{itemize}
  \item Consider the poset of subfaces of $\ P \ $ and its associated M$\ddot{\mbox{o}}$bius function $\mu$ \cite{ro}.
  The following identities hold:
  $$ |P \cap \mathbb{Z}^{d}| \  =  \sum_{F \ \mbox{subface of }  P}|F^{\circ} \cap \mathbb{Z}^{d}|  \ \ \ \  \mbox{and} \ \ \ \
  |P^{\circ} \cap \mathbb{Z}^{d}|  \  =  \sum_{F\ \mbox{subface of } P}\mu(F,P)|F \cap \mathbb{Z}^{d}|. $$
\item   Assume that $\ P \ $ has dimension $\ k \ $ and its vertices lie in $\ \mathbb{Z}^{d}. \ $ Erhart's theory, see Diaz and Robins \cite{dr}, Erhart \cite{E}, and Macdonald \cite{mac}, tell us that the functions
$$ \ P(n) \ = \ |nP \cap \mathbb{Z}^{d}| \  \ \ \ \ \mbox{and} \ \ \ \ \ Q(n) \ = \ |nP^{\circ} \cap \mathbb{Z}^{d}| \ $$
are polynomials of degree $\ k \ $ such that $\  P(n)  \ = \  (-1)^k Q(-n).\ $ Moreover, the degree $\ k \ $ coefficients of both $\ P(n) \ $ and $\ Q(n)\ $ are equal to the volume of $\ P$:
$$  \lim_{n \rightarrow \infty} \frac{P(n)}{n^k} \ \ = \ \ \lim_{n \rightarrow \infty} \frac{Q(n)}{n^k} \ \ = \ \ \mathrm{ vol}(P).$$

\item Suppose that the convex polytope $\ P \subseteq\mathbb{Z}^{d}\ $ has dimension $\ d, \ $  integer vertices,  and $d$-edges emanating from each vertex of $P$ which generate $\mathbb{Z}^{d}$. To this data one associates a toric variety $X_P$ and an holomorphic line bundle $\ L_P \rightarrow X_P \ $  such that:
      $$|P \cap \mathbb{Z}^{d}|  \ = \ \chi(L_P) \ = \ \mathrm{dim}\ H^0(X_P, L_P).$$ Thus $\ |P \cap \mathbb{Z}^{d}|\ $ counts independent sections of the line bundle $\ L_P \rightarrow X_P. \ $ The standard reference for this result is Danilov \cite{da}.

  \item The construction above can be understood in terms of symplectic manifolds and geometric quantization, see Guillemin \cite{gui},  Guillemin, Ginzburg, and Karshon \cite{ggk}, Hamilton \cite{ham}.  Deltzant \cite{del}  constructed a toric symplectic manifold $X_P$, via symplectic reduction, which comes with a K$\ddot{\mbox{a}}$hler structure and a pre-quantum line bundle $\ L_P, \ $ in the sense that first Chern class of  $\ L_P \ $   is the symplectic form of $\ X_P$. The holomorphic structure on  $\ X_P\ $ give rise to a polarization on $\ L, \ $ therefore $\ H^0(X_P, L_P)\ $ is the Hilbert space associated to $\ X_P \ $ in the  geometric quantization approach, see  \'Sniatycki \cite{sn}   and  Woodhouse \cite{w}.

  \item It follows from the  Duistermaat-Heckman theorem, see \cite{dh} and Guillemin \cite{gui},  that the phase space symplectic manifold $\ X_P\ $ and the convex polytope $\ P \ $ have the same volume. Therefore, in this case, the transition
      $$  \mathrm{vol}(P)  =  \mathrm{vol}(X_P) \ \ \ \ \ \ \ \longrightarrow \ \ \ \ \ \ \  |P \cap \mathbb{Z}^{d}|   =  \mathrm{dim}\ H^0(X_P, L_P) $$
is a numerical manifestation of the classical-to-quantum transition:
   $$X_P \ \ \ \ \ \ \ \longrightarrow \ \ \ \ \ \ \ H^0(X_P, L_P). \ $$

 \item A different sort of relation between $|P \cap \mathbb{Z}^{d}|$ and the volume of polytopes arises by considering the polytopal deformation $\ P_h \ $ of $\ P \ $ defined by deforming the equations defining $\ P \ $ by adding a small number $\ h \ $ to the constant term of each equation. Still with the same conditions on the polytope $\ P \ $ as above we have that:
      $$ |P\cap \mathbb{Z}^{d}| \  =  \ \mathrm{Todd}(P,\frac{\partial}{\partial h})\mathrm{ vol}(P_h)\bigg|_{h=0} \ \  \ \ \ \mbox{and} \ \ \ \ \ |P^{\circ} \cap \mathbb{Z}^{d}| \  =  \ \mathrm{Todd}(P,-\frac{\partial}{\partial h})\mathrm{ vol}(P_h)\bigg|_{h=0},$$
      where $\ \mathrm{Todd}(P,\frac{\partial}{\partial h}) \ $ is the operator obtained by substituting  $\ \frac{\partial}{\partial h}\ $ into an explicitly defined formal power series introduced by Todd. This result is due to  Khovanskii and Pukhlikov \cite{kp} and has  been extended, using different techniques,  to the case of simple polytopes ($d$-edges emanating from each vertex) by  Brion and Vergne \cite{brve},   Cappell and Shaneson \cite{cap}, Guillemin, Ginzburg, and Karshon \cite{ggk}, and Karshon, Sternberg,  and Weitsman
      \cite{ksw}.

\item Another approach -- applicable for a rational convex polytope $P$ -- relates   $|P \cap \mathbb{Z}^{d}|$ with the volume of the various faces of $P$:
$$ |P \cap \mathbb{Z}^{d}| \ \ = \ \ \sum_{F\ \mbox{subface of } P}c(F,P) \mathrm{ vol}(F)$$
where the coefficients $\ c(F,P)\ $ are rational numbers which satisfy the properties of being local and computable, with the measure on faces defined in terms of the lattice generators of the affine extension of each face. This result is due
to Berline and Vergne \cite{blve}.

   \item The counting of lattice points inside a polytope has a long history which we do not attempt to summarize, the interested reader may consult  Brion \cite{br}, De Loera \cite{dlo}, Lagarias and  Ziegler \cite{lag}, and the references therein.  We remark that a polynomial time algorithm for such counting was introduced by Barvinok \cite{bar}.

\end{itemize}
Our construction of continuous analogues for certain type of combinatorial objects is best explained via the following flow
diagram:

$$ \mbox{combinatorial object}  \ \ \    \longrightarrow \ \ \   \mbox{lattice path reformulation}  \ \  \  \longrightarrow   $$
$$\ \ \  \mbox{finite decomposition over patterns}  \ \ \  \longrightarrow \ \ \ \mbox{volume of polytopes} \ \ \  \longrightarrow  \ \ \   $$
$$\mbox{countable decomposition over patterns}   \ \ \ \longrightarrow \ \ \  \mbox{continuous object.}$$


In this work we consider problem II for the binomial and Catalan numbers applying the methodology outlined above and described in details
in Sections 2 and 3. In both cases we begin by decomposing the given sequence of numbers as finite sums over time and patterns, where each summand counts the interior points of a lattice polytope.  The starting point to achieve this decomposition is to describe our given sequence of numbers as counting lattice paths, e.g. the Catalan numbers count Dyck paths. Once we have an interpretation of each summand as counting interior points of convex polytopes, we define our continuous analogous by removing the integrality restrictions, i.e. we compute volume of polytopes and replace finite sums by countable sums. The construction of continuous analogues for the binomial coefficients leads to the development of a continuous analogue for the discrete binomial distribution whose density is shown in Figures 2 and 3.\\

Our constructions can be motivated from a physical point of view as
follows. Ever since Feynman reformulated quantum mechanics in
terms of path integrals \cite{f}, constructing a rigorous theory
for such integrals has been a major challenge for mathematicians.
Counting (weighted) lattice paths may be regarded as a fully
discretized version of this problem. Our proposal -- from this viewpoint --  is to extend the
domain of allowed paths:
$$ \mbox{lattice\  paths}  \ \ \  \subseteq  \ \ \ \mbox{directed paths} \ \ \ \subseteq \ \ \ \mbox{continuous paths} ,$$
from lattice paths to directed paths, which form a moduli space and yet by construction retain a strong combinatorial flavor. \\

We stress that we are after continuous analogues rather than extensions to continuous variables. It is well-known that the latter can be achieved for the binomials -- and thus for the Catalan numbers -- with the help of the classical gamma and beta functions. This work takes part in our program aimed to bring geometric methods to the study of problems arising from the theory of complex networks \cite{cd,d,dg,dv}.

\section{Lattice Paths and Patterns}\label{2}

Let us recall the settings upon which the theory of lattice paths is built \cite{fla}. We fix throughout Sections \ref{2} and \ref{3} the following data: a dimension $\ d \in \mathbb{N}_{>0},\ $ and  step vectors $\ V=\{ v_1,...,v_k\} \ \subseteq \  \mathbb{Z}^d \ \subseteq \ \mathbb{R}^d $
with index set $\ [k] =\{1,..,k\}$.\\

 A lattice path (with steps in $V$) from $\ p \ $ to $\ q \ $ in $\ \mathbb{Z}^d\ $
is given by a tuple of points $\ (p_0,...,p_n)\ $ in $\ \mathbb{Z}^d, \ $ with $\ n\geq 1, \ $  such that, for  $\ i \in [n],\ $ we have
$\ \ p_0  =  p, \ \ \ p_n =  q, \ \ \  p_i  -  p_{i-1} \ \in \ V. \ $
Thinking of time as a discrete variable,  we identify the parameter $\ n \ $ with the travel time for a particle starting at $\ p \ $ and moving towards $\ q \ $ trough the  path $\ (p_0,...,p_n).\ $ We add a zero time path from each lattice point to itself; this convention turns the set of lattice points and lattice paths among them into the objects and morphisms, respectively, of a category with composition given by concatenation.\\

The set  $\ \mathrm{L}_{p,q}\ $ of lattice paths from $\ p \ $ to $\ q, \ $ can be described as
$$\mathrm{L}_{p,q} \ = \ \bigsqcup_{l =0}^{\infty}\mathrm{L}_{p,q}(l) \ = \
\bigsqcup_{l =0}^{\infty}\big\{(a_1,...,a_l) \in [k]^l \ \ \big| \ \  p +v_{a_1}+ \cdots + v_{a_l}  = \ q \big\}, $$
where $ \ \mathrm{L}_{p,q}(l) \ $ is the set of lattice paths from $\ p \ $ to $\ q \ $ displayed in time $\ l.$\\

A pattern (of directions) of length $\ n +1 \ $  on the set of indexes $\ [k]\ $ is given by a $\ (n+1)$-tuple $\ c=(c_0, c_1,...,c_n)\ $ such that $\ c_i \in [k] \ $ and $\ c_i \neq c_{i+1}$.  Let $\ D(n,k) \ $ be the set of all patterns of
length $\ n +1. \ $ We have a natural map that associates a pattern to each lattice path by contracting contiguous repeated indices.  For example  set $\ n=8, \   k=3, \ $ and consider a lattice path $\ (2,2,1,3,3,3,1,1) \in \mathrm{L}_{p,q}(8).\ $ The  associated pattern is $\ (2,1,3,1) \in D(3,3).$ We formally add a pattern of length $0$.\\

Going back to our general settings, we have that
$$\ \ \ \displaystyle \mathrm{L}_{p,q} \ = \ \bigsqcup_{l=0}^{\infty} \bigsqcup_{n=0}^{l-1} \bigsqcup_{c\in D(n,k)}\mathrm{L}_{p,q}^c(l), \ \ \ $$
where $\ \mathrm{L}_{p,q}^c(l) \ $ is the set of lattice paths from $\ p \ $ to $\ q \ $ displayed in time $\ l \ $ and with associated pattern equal to $\ c  \in  D(n,k)$.\\

\begin{prop}{\em \
\begin{enumerate}
  \item The number of lattice paths from  $\ p \ $ to $\ q \ $ displayed in time $\ l \ $ is given by
$$|\mathrm{L}_{p,q}(l)| \ \ =  \ \  \sum_{n=0}^{l-1}\sum_{c\in D(n,k)}|\mathrm{L}_{p,q}^c(l)|,$$
where $\ \mathrm{L}_{p,q}^c(l) \ $ is the set of tuples $\ (s_0,...,s_n)\in \mathbb{N}_{>0}^{n+1} \ $ such that
$$p + s_0v_{c_0} + \cdots + s_nv_{c_n} =  q \ \ \ \ \  \mbox{and} \ \ \ \ \ s_0 + \cdots + s_n =  l.$$
  \item The number of lattice paths  from  $\ p \ $ to $\ q, \ $ if finite, \ is given by
$$|\mathrm{L}_{p,q}| \ \ =  \ \  \sum_{l=0}^{\infty}\sum_{n=0}^{l-1}\sum_{c\in D(n,k)}|\mathrm{L}_{p,q}^c(l)|.$$
\end{enumerate}

}
\end{prop}

The main problem in lattice path theory is to count the number of lattice paths joining a pair of lattice points. Usually further restrictions are imposed on the allowed paths. For example, one may want to count lattice paths that are restricted to visiting points in a subset of $\ \mathbb{Z}^d, \ $ which we assume to be the set of integral points of a convex polyhedron $\ H \subseteq\mathbb{R}^d. \ $
let $\ \mathrm{L}_{p,q}^{H}(l)\ $ be the set of time $\ l \ $ lattice paths lying in $\ H.$
Also let $\ \mathrm{L}_{p,q}^{c,H}(l)\ $ be the set of lattice paths fully included in  $\ H \ $ of time $ \ l, \ $ and pattern $ \ c. \ $ \\

\begin{prop}{\em Let $\ p,q \in H.$
\begin{enumerate}
  \item The number of lattice paths from  $\ p \ $ to $\ q \ $ fully included in the convex polyhedron $\ H \ $ and displayed in time $\ l \ $ is given by $$|\mathrm{L}_{p,q}^H(l)| \ \ =  \ \  \sum_{n=0}^{l}\sum_{c\in D(n,k)}|\mathrm{L}_{p,q}^{c,H}(l)|,$$
where $\ \mathrm{L}_{p,q}^{c,H}(l)\ $ is the set of tuples $\ (s_0,...,s_n)\in \mathbb{N}_{>0}^{n+1}\ $ such that the following conditions hold for $\ 0 \leq i \leq n-1: \ $
$$p + s_0v_{c_0} + \cdots + s_iv_{c_i} \in H, \ \ \ \ \ \ p + s_0v_{c_0} + \cdots + s_nv_{c_n} =  q, \ \ \ \ \ \ s_0 + \cdots + s_n =  l.$$

  \item The number of lattice paths from  $\ p \ $ to $\ q \ $ fully included in the convex polyhedron $\ H, \ $  if finite,\ is given by $$|\mathrm{L}_{p,q}^H| \ \ =  \ \ \sum_{l=0}^{\infty} \sum_{n=0}^{l-1}\sum_{c\in D(n,k)}|\mathrm{L}_{p,q}^{c,H}(l)|.$$
\end{enumerate}
}
\end{prop}

\section{From Lattice Paths to Directed Paths}\label{3}

Our next goal is to provide a suitable setting for "counting" directed paths, for  which we  keep the same set of allowed
directions $\ V  =  \{v_1,...,v_k \} \subseteq \mathbb{Z}^k \ $ as for lattice paths,
while  lifting the discrete time restriction, i.e. we consider  tuples
$\ (p_0,...,p_n) \in \mathbb{R}^d \ $ such that
$$\ \ p_0 \ = \ p, \ \ \ \ \ p_n\ = \ q, \ \ \ \ \ p_i \ = \ p_{i-1} + s_i v , \ \ \ \ \mbox{with}\ \  v \in V, \ \
\mbox{and} \ \ s_i \in \mathbb{R}_{\geq 0}. \ $$  The total travel time   $\ t\ = \ s_1 \ + \ \cdots \ + \ s_n\ $ no longer has to be an integer; hence
we are facing a moduli space of paths rather than a discrete set of paths.\\

To formalize the "counting" of such paths we turn to  our work on indirect influences on directed manifolds \cite{cd}.
Essentially this approach  give us a way to put measures on the various  components of the space of directed paths on a directed manifold. Fortunately, for our present purposes, we can proceed quite independently in an essentially self-contained fashion.\\

Recall from \cite{cd}\  that a directed manifold is a smooth manifold together with a tuple of vector fields on it. \  We are going to work with the directed manifold $\ (\mathbb{R}^d, v_1,...,v_k). \ $   A directed path from $\ p \in \mathbb{R}^d \ $ to $\ q \in \mathbb{R}^d \ $ displayed in time $\ t>0\ $ and going through  $\ n\geq 0\ $ changes of directions  is  parameterized by a pair $\ (c,s)\ $ with the following properties:

\begin{itemize}
\item $c=(c_0,...,c_n)\ $ is a pattern in $\ D(n,k).$

\item $s=(s_0,...,s_n)\ $  is  a $\ (n+1)$-tuple such that $\ s_0+\cdots+s_n=t, \ $  with $\ s_i \in \mathbb{R}_{\geq 0\mathbb{}}.\ $ We say that $\ s \ $ defines the time distribution of the directed path associated to $\ (c,s), \ $ and let $\ \Delta_{n}^t \ $ be the $\ n$-simplex of all such tuples. We regard $\ \Delta_{n}^t \ $ as a subset of the space $\ \mathbb{R}^{n+1}\ $ endowed with its canonical inner product.

\item  $\ (c,s)\ $ determines a $\ (n+2)$-tuple of points $ \ (p_0,\ldots,p_{n+1})\ \in \ (\mathbb{R}^d)^{n+2} \ $
given by: $$\ p_0\ = \ p \ \ \ \ \ \mbox{and} \ \ \ \ \ p_{i}\ = \ p_{i-1} \ + \ s_{i-1}v_{c_{i-1}}  \ \ \ \ \mbox{for}\ \ \ \ 1 \leq i \leq n+1.$$

\item $\ (c,s) \ $ must be such that $\ p_{n+1}(c,s) \ = \ q.$

\end{itemize}

The pair $\ (c,s)\ $ determines the directed polygon path  $$\varphi_{c,s}:[0, s_0 + \cdots + s_n] \ \ \simeq \ \ [0,s_0] \ \underset{s_0,0}{\bigsqcup}\ \cdots \ \underset{s_{n-1},0}{\bigsqcup} \ [0,s_n] \ \ \longrightarrow \ \ M  $$
from $\ p \ $ to $\ q \ $ where the restriction of $\ \varphi_{c,s}\ $ to the interval $\ [0,s_i] \ $ is given by
$$\varphi_{c,s}|_{[0,s_i]}(r) \ \ = \ \ p_{i} \ + \  rv_{c_i}  \ \ \ \ \ \mbox{for } \ \ \ r\in [0,s_i].$$
We say that the points $\ p_0,...,p_n\ $ are the peaks of the path $\ \varphi_{c,s}. $

The moduli space $\ \Gamma_{p,q}(t)\ $ of directed paths from $\ p \ $ to
$\ q \ $ displayed in time $\ t>0 \ $ is given by
$$\Gamma_{p,q}(t)\ \ = \ \ \coprod_{n=0}^{\infty}\coprod_{c \in D(n,k)} \{ s \in \Delta_{n}^{t} \ | \ p_{n+1}(c,s)  =  q\} \ \ = \ \
\coprod_{n=0}^{\infty}\coprod_{c \in D(n,k)} \Gamma_{p,q}^c(t). $$ In addition we formally set
$$\Gamma_{p,q}(0) \ \ = \ \ \Gamma_{p,q}^{\emptyset}(0) \ \ = \ \
\left\{
\begin{array}{lcl}
\{p\} \ \ \ \   \mbox{if} \ \ p=q,\\
& & \\
\ \emptyset \ \ \ \ \   \mbox{ otherwise}.
\end{array}
\right.$$ The unique path in the latter set has the empty pattern. \\

Our guiding principle in this work is that one can think of the space $\ \Gamma_{p,q}(l) \ $ as being a continuous analogue
of the set $\ \mathrm{L}_{p,q}(l) \ $ of lattice paths from $\ p \ $ to $\ q \ $ displayed in time $\ l \geq 0. \ $  Note that in $\ \Gamma_{p,q}(t) \ $ neither $\ p \ $ nor $\ q \ $ nor $\ t \ $ are restricted to be integers points. Even if they are integers $\ \Gamma_{p,q}(t) \ $ is still a  larger space than $\ \mathrm{L}_{p,q}(t) \ $ since in $\ \Gamma_{p,q}(t) \ $ the intermediary peaks are not restricted to be integer points.\\

\begin{prop}\label{p}
{\em Consider the directed manifold $\ (\mathbb{R}^d, v_1,...,v_k), \ $ let $\ H  \ \subseteq \ \mathbb{R}^d \ $ be a convex polyhedron, let
$\ p,q \in H, \ $ and $\ t \in \mathbb{R}_{>0}.$
\begin{enumerate}
  \item For $\  c=(c_0,...,c_n) \in D(n,k), \ $ the space $\ \Gamma_{p,q}^c(t) \ $ is the convex polytope given by:
$$\Gamma_{p,q}^c(t) \ = \ \bigg\{(s_0,...,s_n)\in \mathbb{R}_{\geq 0}^{n+1} \ \bigg| \
p+ s_0v_{c_0} + \cdots + s_nv_{c_n} =  q, \ \ s_0 + \cdots + s_n =  t  \bigg\} .$$

\item $ \mathrm{L}_{p,q}^c(t) \ $ is the set of integer points in the interior of $\ \Gamma_{p,q}^c(t). \ $

\item For $\  c\in D(n,k), \ $ let $\ \Gamma_{p,q}^{c,H}(t) \ $ be the subset of $ \ \Gamma_{p,q}^{c}(t) \ $ consisting points whose associated
path $\ \varphi_{c,s}\ $ lies entirely  in $\ H. \ $
The moduli space $ \ \Gamma_{p,q}^{c,H}(t) \ $ is the convex polytope consisting  of all
tuples $\ (s_0,...,s_n)\in \mathbb{R}_{\geq 0}^{n+1} \ $ such that the following conditions hold for $\ 0 \leq i \leq n-1: \ $
$$p + s_0v_{c_0} + \cdots + s_iv_{c_i} \in H, \ \ \ \ \ \ p + s_0v_{c_0} + \cdots + s_nv_{c_n} =  q, \ \ \ \ \ \ s_0 + \cdots + s_n =  l.$$

\item  $\ \mathrm{L}_{p,q}^{c,H}(t) \ $ is the set of integer points in the interior of $\ \Gamma_{p,q}^{c,H}(t). \ $
\end{enumerate}
}
\end{prop}

Next we introduce our main definition in this work.

\begin{defn}\label{vol}
{\em
The volume of the moduli space of directed paths $ \ \Gamma_{p,q}(t) \ $ is given by:
$$\mathrm{vol}(\Gamma_{p,q}(t)) \ \ = \ \ \sum_{n=0}^{\infty}\sum_{c \in D(n,k)} \mathrm{vol}(\Gamma_{p,q}^c(t)) .$$
}
\end{defn}

\begin{rem}{\em
To compute the volume of a polytope $ \ P \subseteq  \ \Delta_{n}^t \ $ we regard it as
a top dimensional subset of its affine linear span $\  \widehat{P} \ \subseteq \ \mathbb{R}^{n+1}, \ $ and compute its volume with respect to the Lebesgue measure on $\  \widehat{P} \ $ induced by the inner product on $\ \mathbb{R}^{n+1}. \ $
As it stands, there is no guarantee  that the infinite sum above is convergent. Nevertheless, it turns out to be convergent in  the examples developed in Sections 4 and 5.}
\end{rem}

\section{Continuous Binomials Coefficients}\label{4}

The binomial coefficient $ {m \choose n}  $ counts sets of cardinality $ n  $ within a set of cardinality $\ m. $   To construct continuous analogues for the binomial coefficients we need a lattice path representation for them.
Consider the step vectors $$\ V\  = \ \{(1,0), (0,1)\} \ \subseteq \ \mathbb{Z}^2 \ \subseteq  \ \mathbb{R}^2. \ $$  It is well-known that  the binomial coefficient $\ {m \choose n}\ $
is such that
$${m \choose n} \ \ = \ \  \big|\big\{\mbox{lattice \ paths \ from } \ (0,0) \ \ \mbox{to}\ \ (n,m-n) \big\}\big|.$$
Notice that such a lattice path is displayed in time $\ m. \  $
Thus according to our general methodology for constructing continuous analogues  we  should consider the directed manifold $\ (\mathbb{R}^2, (1,0), (0,1)),\ $  and compute the volume of  the moduli spaces of directed paths $\ \Gamma_{(0,0),(x,y)}(t); \ $ we denote the latter space by $\ \Gamma(x,y)\ $ as it is empty unless $\ t=x+y. \ $

\begin{defn}\label{c}
{\em
For $  0 <  s < x\ $  in $\ \mathbb{R}, $ the continuous binomial coefficient $\ {x \brace s}\ $ is given  by
$${x \brace s} \ \ = \ \ \mathrm{vol}(\Gamma(s,x-s)) \ \ = \ \ \sum_{n=0}^{\infty}\sum_{c \in D(n,k)} \mathrm{vol}(\Gamma^c(s,x-s)) .$$
The domain of the symbol $\ {x \brace s}\ $  is extended by continuity to $\  0 \leq  s \leq  x. \ $   }
\end{defn}

Intuitively, $\ {x \brace s}\ $  is the total measure of the set of paths starting
at the origin and built with horizontal and vertical moves, with
travelling time in the horizontal direction of $\ s, \ $ and
travelling time in the vertical direction of $\ x-s. \ $ Figure 1
shows a couple of directed paths accounted for by the continuous binomial coefficient $\ {\pi \brace e}.$
\begin{center}
\psset{unit=3cm}
    \begin{pspicture}(0,-0.4)(3.5,0.5)
        \psline[linewidth=1pt]{->}(0,0)(1.2,0)
        \rput(1,-0.1){$1$}
        \rput(2,-0.1){$2$}
        \rput(3,-0.1){$3$}
        \rput(-0.15,0.3){0.3}
        \psline[linewidth=1pt,linestyle=dotted]{}(0,0)(0,0.6)
        \psline[linewidth=1pt,linestyle=dotted]{}(0,0)(3.2,0)
        \psline[linewidth=1pt]{->}(1.2,0)(1.2,0.2)
        \psline[linewidth=1pt]{->}(1.2,0.2)(2.1,0.2)
        \psline[linewidth=1pt]{->}(2.1,0.2)(2.1,0.4)
        \psline[linewidth=1pt]{->}(2.1,0.4)(2.7182,0.4)
        \psline[linewidth=1pt]{->}(0,0)(0,0.3)
        \psline[linewidth=1pt]{->}(0,0.3)(2.7182,0.3)
        \psline[linewidth=1pt]{->}(2.7182,0.3)(2.7182,0.4)
        \psdot[dotstyle=*,
    dotsize=4.5pt](0,0)
    \psdot[dotstyle=*,
    dotsize=4.5pt](2.7182,0.4)
    \rput(3.05,0.5){$(e,\pi-e)$}
       \rput(1.5,-0.3){Figure 1. Directed paths in $\ \Gamma(e,\pi-e)\ $ with patterns $\ (1,2,1,2,1)\ $ and $\ (2,1,2).$}
    \end{pspicture}
\end{center}

To compute explicitly the continuous binomial coefficients we used the following identity shown in \cite{cd}.

\begin{iden}\label{iden}
{\em The volume $\ \mathrm{vol}(\Gamma(x,y))\ $  of the convex polytope $\ \Gamma(x,y) \ $ is given by
$$\sum_{n=0}^\infty\Big(2\frac{x^{n}y^n}{n!n!}\ + \ \frac{x^{n+1}y^{n}}{(n+1)!n!}\ +\ \frac{x^{n}y^{n+1}}{n!(n+1)!}\Big) \ \ = \ \ 2\sum_{n=0}^\infty \frac{x^{n}y^n}{n!n!}\  +  \ (x+y)\sum_{n=0}^\infty\frac{x^{n}y^{n}}{n!(n+1)!}.$$
}
\end{iden}

\begin{rem}{\em
The presence of a couple of factorials in the denominators of the summands in formula for $\ \mathrm{vol}(\Gamma(x,y))\ $  from the proof of Theorem \ref{cn}  guarantees uninform convergency. Similar remarks will apply for all the power series appearing in this section.
}
\end{rem}

\begin{thm}\label{cn}
{\em For  $\  0 \leq  s \leq  x\ $ the continuous binomial function $\ {x \brace s}\ $ is  given by
$$2\sum_{0 \leq a \leq b, \ a+b\  \mathrm{even}}(-1)^{\frac{b-a}{2}}{b \choose \frac{a+b}{2}}\frac{x^a}{a!}\frac{s^b}{b!} \ \ + \ \
\sum_{0 \leq a \leq b-1, \ a+b\
\mathrm{odd}}(-1)^{\frac{b-a-1}{2}}{b \choose
\frac{a+b+1}{2}}\frac{x^a}{a!}\frac{s^b}{b!} \ \ \ +  $$
$$\sum_{0 \leq a \leq b+1, \ a+b\  \mathrm{odd}}(-1)^{\frac{b-a+1}{2}}{b \choose \frac{b-a+1}{2}}\frac{x^a}{a!}\frac{s^b}{b!} \ \ \ =$$
$$ 2\sum_{n=0}^\infty \frac{s^{n}(x-s)^n}{n!n!}\  +  \ x\sum_{n=0}^\infty\frac{s^{n}(x-s)^{n}}{n!(n+1)!}\ \ = \ \
\sum_{n=0}^\infty(x+2n+2)\frac{s^{n}(x-s)^{n}}{n!(n+1)!}.$$
}
\end{thm}

\begin{proof}
Follows from Identity \ref{iden} using that   $\ {x \brace s} \  = \  \mathrm{vol}(\Gamma(s,x-s)). \ $
\end{proof}

The following result gives continuous analogues for a couple of  well-known properties of the binomial coefficients.

\begin{cor}\label{by}
{\em For $0 \leq s \leq x,$ we have that $ \ \displaystyle {x \brace 0}  =  {x \brace x} =   2  +  x, \ \  $ and  $\ \displaystyle {x \brace s}  =  {x \brace x-s}. $
}
\end{cor}

\begin{rem}{\em  We have chosen to set the value of $ \ {x \brace 0} $ by continuity. The increment of  the weight of the unique directed path joining $\ (0,0) \ $ and $\ (x,0)\ $ from $\ 1 \ $ to $\ 2+x \ $ is reminiscent of the process by which a particle acquires an higher effective mass -- compared to its bare mass -- by being surrounded by other massive particles.
}
\end{rem}

\begin{cor}{\em For $\ s \geq 0, \ $ we have that:
$$\ {1 +s \brace s} \ \ = \ \ 3 \ + \ \sum_{n=1}^\infty \left( n^2 + 3n +3\right)\frac{s^{n}}{n!(n+1)!} .$$
}
\end{cor}

Our next results shows that the binomial coefficients are an eigenfunction, with eigenvalue 1, of an hyperbolic partial differential
equation.

\begin{thm}\label{pbc11}
{\em The continuous binomial coefficients $\ {x \brace s} \ $ satisfy the following partial differential equation:
$$\Big( \frac{\partial^2}{\partial x^2}  \ + \  \frac{\partial^2}{\partial x \partial s}  \Big){x \brace s} \ = \ {x \brace s}.$$
}
\end{thm}

\begin{proof}

We know from \cite{cd} that the function $\ \mathrm{vol}(\Gamma)=
\mathrm{vol}(\Gamma(x,y)), \ $ given explicitly in the proof of
Theorem \ref{cn}, \ satisfies the following  partial differential
 equation:
$$\frac{\partial^2 \mathrm{vol}(\Gamma)}{\partial x \partial y}  \ = \ \mathrm{vol}(\Gamma).$$
Therefore we have that:
$$\Big( \frac{\partial^2}{\partial x^2}  \  + \  \frac{\partial^2}{\partial x \partial s}  \Big){x \brace s}  \ = \
\frac{\partial^2 \mathrm{vol}(\Gamma(s,x-s)) }{\partial x^2}  \ \ + \ \  \frac{\partial^2 \mathrm{vol}(\Gamma(s,x-s)) }{\partial x \partial s} \ \ =  $$
$$\frac{\partial^2 \mathrm{vol}(\Gamma)}{\partial y^2}(s,x-s) \ \ + \ \
\frac{\partial}{\partial x}\Big[\frac{\partial \mathrm{vol}(\Gamma)}{\partial x }(s,x-s) \ \ - \ \
\frac{\partial \mathrm{vol}(\Gamma)}{\partial y }(s,x-s) \big] \ \ = $$
$$\frac{\partial^2 \mathrm{vol}(\Gamma)}{\partial y^2}(s,x-s) \ \ + \ \
\frac{\partial^2 \mathrm{vol}(\Gamma)}{\partial x \partial y}(s,x-s)\ \ - \ \
\frac{\partial^2 \mathrm{vol}(\Gamma)}{\partial y^2 }(s,x-s) \ \ = $$
$$\mathrm{vol}(\Gamma(s,x-s))  \ = \  {x \brace s}.$$
\end{proof}

The following result, due to  Tom Koornwinder, gives an explicit formula  for the continuous binomial coefficients in terms
of the modified Bessel functions $\ \mathrm{I}_0 \ $ and $\ \mathrm{I}_1.$

\begin{thm} {\em  For $\ 0 \leq s \leq x, \ $ we have that
$${ x \brace s } \ \ = \ \  2 \mathrm{I}_0(2 s^{1/2}(x-s)^{1/2})\ + \ \frac{x}{s^{1/2}(x-s)^{1/2}}  \mathrm{I}_1(2 s^{1/2}(x-s)^{1/2}).$$
}
\end{thm}

\begin{proof}
The result follows from Identity \ref{iden} and the defining expressions
$$ \mathrm{I}_0(z) \ = \ \sum_{k=0}^{\infty}\frac{(\frac{1}{4}z^2)^{k}}{k!k!}  \ \ \ \ \ \ \mbox{and} \ \ \ \ \ \
\mathrm{I}_1(z) \ = \ \frac{z}{2}\sum_{k=0}^{\infty}\frac{(\frac{1}{4}z^2)^{k}}{k!(k+1)!}$$
for the the modified Bessel functions $\ \mathrm{I}_0 \ $ and $\ \mathrm{I}_1.$

\end{proof}

Fix $\ 0 < p < 1 \ $ and $\ x > 0. \ $ A continuous analogue of the discrete binomial distribution may be defined via the density function
on the interval $\ [0,x] \ $ given by
$$\frac{1}{b_p(x)}{ x \brace s }p^s(1-p)^{x-s} \ \ \ \ \ \mbox{where} \ \ \ \ \ b_p(x)=\int_0^x{ x \brace s }p^s(1-p)^{x-s}ds$$
which, intuitively, measures the probability of the motion of a particle in such that:
\begin{itemize}
  \item The particle starts at $\ (0,0) \ $ and moves with speed $\ 1 \ $ for $\ x \ $ units of times.
  \item The particle moves with probability $\ p \ $ in the horizontal direction, and  with probability $\ 1-p \ $ in the vertical direction.
  \item The particle ends up at the point $\ (s,x-s).$
\end{itemize}

For $\ p = \frac{1}{2} \ $ the normalization constant, up to a factor of $\ 2^{-x}, \ $ is given by
$$ \int_0^x{ x \brace s }ds,$$
and measures the number of paths starting at the origin and traveling for $\ x \ $ units of time either in the horizontal or in the vertical direction. Note that $\  \int_0^x{ x \brace s }ds \ $ plays, for the continuous binomial coefficients, the  role that $\ 2^n \ $ plays for the binomial coefficients.

Below we are going to use the following identity, valid for $\ n,m \in \mathbb{N}, \ $ involving the classical beta $\ \mathrm{B}  \ $ and gamma $\ \Gamma \ $ functions:
$$\int_0^1s^n(1-s)^mds\ = \ \mathrm{B}(n+1,m+1)\ = \ \frac{\Gamma(n+1)\Gamma(m+1)}{\Gamma(n+m+2)}
\ = \ \frac{n! m!}{(n+m+1)!}.$$

Also we are going to use the falling factorials $\ (a)_n \ = \ a(a-1)\cdots (a-n+1)\ $ for $\ a \geq n. \ $ Note that the notation $\ a^{\underline{n}} \ $ for the falling factorial is also
quite common.

\begin{thm}\label{b}
 {\em  \
\begin{enumerate}
  \item For $\ x \in \mathbb{R}_{\geq 0} \ $ we have that
$$\int_0^x{ x \brace s }  ds= 2(e^x-1) \ \ \ \ \ \mbox{and \ thus} \ \ \ \ \ b_{\frac{1}{2}}(x)  =  (e^x-1)2^{1-x}.$$

\item The following identities hold:
$$ e \ = \  1 \ + \ \frac{1}{2}\int_{0}^{1} { 1 \brace s } ds \ \ \ \  \ \ \ \  \mbox{and} \ \ \ \ \ \ \ \
\lim_{n \rightarrow \infty} \bigg(1  +   \frac{x}{n}\bigg)^n \ = \ 1 \ + \ \frac{1}{2}\int_{0}^{x} { x \brace s } ds .$$
  \item For $\ 0<p<1 \ $ we have that $\ b_p(x) \ $ is given by
  $$2(1-p)^x\sum_{k,n=0}^\infty \mathrm{ln}^k(\frac{p}{1-p}){n+k \choose k}\bigg[\frac{x^{2n+k+1}}{(2n+k+1)!}\  + \   \frac{x^{2n+k+2}}{(2n+2)(2n+k+1)!}\bigg].$$

  \item  For $\ 0<p<1, \ $the function $\ b_p(x) \ $ can  be written as
 $$\pi^{\frac{1}{2}}(p(1-p))^{\frac{x}{2}}\sum_{n=0}^\infty \frac{(x+2n+2)}{(n+1)!} \big(\frac{x}{\mathrm{ln}(\frac{p}{1-p})}\big)^{n+\frac{1}{2}}
I_{n+\frac{1}{2}}\big(\frac{x}{2}\mathrm{ln}(\frac{p}{1-p})\big).$$

\end{enumerate}
}
\end{thm}

\begin{proof} Item 2 follows directly from item 1, which is shown as follows:
$$\int_0^x{ x \brace s } ds\ = \ 2\sum_{n=0}^\infty
\frac{1}{n!n!}\int_0^xs^{n}(x-s)^nds \ + \ \sum_{n=0}^\infty\frac{x}{(n+1)!n!}
\int_0^x s^{n}(x-s)^{n} ds \ = $$
$$2\sum_{n=0}^\infty
\frac{x^{2n+1}}{n!n!}B(n+1,n+1) \ + \ \sum_{n=0}^\infty\frac{x^{2n+2}}{(n+1)!n!}
B(n+1,n+1) \ =$$
$$2\sum_{n=0}^\infty
\frac{x^{2n+1}}{(2n+1)!} \ + \ 2\sum_{n=0}^\infty\frac{x^{2n+2}}{(2n+2)!}
 \ = \ 2\mathrm{sinh}(x) + 2\mathrm{cosh}(x)-2\  = $$
$$e^x - e^{-x} +e^x + e^{-x} -2 \ = \ 2(e^x -1).$$
The proof of  item 3 is omitted as it can be deduced just as item 2 of Theorem \ref{zz} below with $\ l=0. \ $ Item 4 follows by termwise integration making the change of variable $\ s  =  \frac{x}{2}(t+1), \ $ using the integral representation for the modified Bessel functions given by the
identity $$ \int_{-1}^{1}(1-t^2)^{v-\frac{1}{2}}e^{zt}dt \ = \ \pi^{\frac{1}{2}}\Gamma(v+\frac{1}{2})(\frac{z}{2})^{-v}I_v(z) .$$ Explicitly, we have that
$$ b_p(x)\ =\ \int_0^x{ x \brace s }p^s(1-p)^{x-s}ds\ = \ (1-p)^x\int_0^x{ x \brace s }e^{\mathrm{ln}(\frac{p}{1-p})s}ds\ \ = $$
$$ (1-p)^x\sum_{n=0}^\infty \frac{x+2n+2}{n!(n+1)!} \int_0^xs^{n}(x-s)^ne^{\mathrm{ln}(\frac{p}{1-p})s}ds\  \ = $$
$$(p(1-p))^{\frac{x}{2}}\sum_{n=0}^\infty \frac{x+2n+2}{n!(n+1)!}(\frac{x}{2})^{2n+1}
 \int_{-1}^{1}(1-t^2)^n e^{\frac{x}{2}\mathrm{ln}(\frac{p}{1-p})t}dt \ \ = $$
$$ \pi^{\frac{1}{2}}(p(1-p))^{\frac{x}{2}}\sum_{n=0}^\infty \frac{x+2n+2}{(n+1)!} \big(\frac{x}{\mathrm{ln}(\frac{p}{1-p})}\big)^{n+\frac{1}{2}}
I_{n+\frac{1}{2}}(\frac{x}{2}\mathrm{ln}(\frac{p}{1-p})). $$
\end{proof}

We define the centered continuous binomial distribution, for $\ p=\frac{1}{2},\ $   via its density function $\ d_x \ $ which has support in the interval $\ [-\frac{x}{2},\frac{x}{2}] \ $ where it is given by $$d_x(s)=  \frac{1}{b(x)}{x \brace  \frac{x}{2} + s}.$$

\begin{figure}[t]
\centering
\includegraphics[scale=0.6]{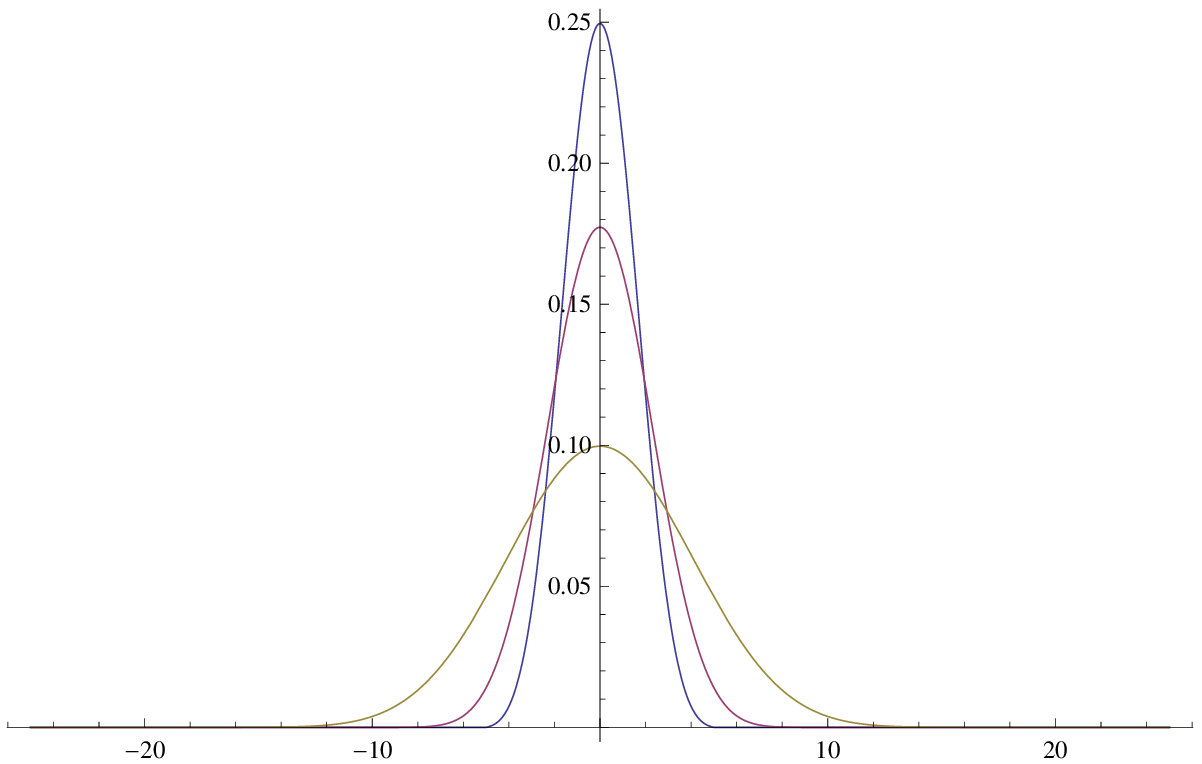}
\captionsetup{labelformat=empty} \caption{Figure 2: \ Continuous binomial density $\ d_x \ $ for $\ x\ =\ 10,\ 20,\ 50.$ }
 \label{fig1}
\end{figure}

\begin{thm}\label{zz}
{\em \
\begin{enumerate}
  \item The moments $\ E_{\frac{1}{2}}(s^l), \ $ for $\ l \geq 1, \ $ of the continuous binomial distribution, for $\ p=\frac{1}{2}, \ $ are given  by
  $$  E_{\frac{1}{2}}(s^l)  \ \  = \ \ \frac{1}{b_{\frac{1}{2}}(x)}\sum_{n=0}^\infty (n+l)_l \frac{x^{2n+l+1}}{(2n+l+l)! } \ \  +\  \
\frac{1}{b_{\frac{1}{2}}(x)}\sum_{n=0}^\infty (n+l)_{(l-1)}\frac{x^{2n+l+2}}{(2n+l+1)!} .$$
  \item The moments $\ E_{p}(s^l), \ $ for  $\ l \geq 1, \ $ of the continuous binomial distribution, for  $\ 0 < p < 1, \ $ are given by
\begin{footnotesize}
 $$\frac{2(1-p)^x}{b_p(x)}\sum_{k,n=0}^\infty \mathrm{ln}^k(\frac{p}{1-p})\binom{n+k+l}{n}(k+l)_l\bigg(\frac{x^{2n+k+l+1}}{(2n+k+l+1)!} \   + \
  \frac{x^{2n+k+l+2}}{(2n+2)(2n+k+l+1)!} \bigg).$$
 \end{footnotesize}
\end{enumerate}

}
\end{thm}

\begin{proof}
Item 1 follows the same pattern as  the proof of Theorem \ref{b}. We show item 2:

$$b_p(x)E_p(s^l)\ = \ \int_0^xs^l{ x \brace s }p^s(1-p)^{x-s} ds\ = $$
$$2\sum_{n=0}^\infty \int_0^xs^l\frac{s^{n}(x-s)^n}{k!n!n!}p^s (1-p)^{x-s}ds\ \ +\ \ \sum_{n=0}^\infty x\int_0^x
s^l\frac{s^{n}(x-s)^{n}}{k!(n+1)!n!} p^s (1-p)^{x-s}ds\ = $$
$$2(1-p)^x\sum_{k,n=0}^\infty \mathrm{ln}^k(\frac{p}{1-p})\int_0^x\frac{s^{n+k+l}(x-s)^n}{k!n!n!}ds\ \ \ \ + \ $$
$$ (1-p)^x\sum_{k, n=0}^\infty \mathrm{ln}^k(\frac{p}{1-p}) x\int_0^x \frac{s^{n+k+l}(x-s)^{n}}{k!(n+1)!n!}  ds\ \ \  = $$
$$2(1-p)^x\sum_{k,n=0}^\infty \mathrm{ln}^k(\frac{p}{1-p})\frac{(n+k+l)!}{k!n!}\frac{x^{2n+k+l+1}}{(2n+k+l+1)!}\ \ \  + \ $$
$$ (1-p)^x\sum_{k, n=0}^\infty \mathrm{ln}^k(\frac{p}{1-p})\frac{(n+k+l)!}{k!(n+1)!} \frac{x^{2n+k+l+2}}{(2n+k+l+1)!} . $$
\end{proof}

\begin{prop}{\em \
\begin{enumerate}
\item The odd moments of $\ d_x \ $ vanish: $\ E(s^{2k+1})  =  0. $

\item The moment $\ E(s^{2k}) \ $ of $\ d_x \ $ is given by
$$E(s^{2k}) \ \ = \ \ \frac{1}{b_{\frac{1}{2}}(x)}\sum_{l=0}^{2k}\binom{2k}{l}(-\frac{x}{2})^{2k-l}\int_{0}^{x}t^{l}{x \brace  t}dt,$$
where $\ \displaystyle \int_{0}^{x}t^{l}{x \brace  t}dt \ $ has been given explicitly in Theorem \ref{zz}.

\item Let $\ \delta\ $ be the Dirac's delta function centered
at $\ 0. \ $ We have that
$$\displaystyle \lim_{x \rightarrow 0} d_x \  =  \ \delta \ \ \ \ \ \mbox{in the weak topology.}$$

\end{enumerate}
}
\end{prop}

\begin{figure}[t]
\centering
\includegraphics[scale=0.8]{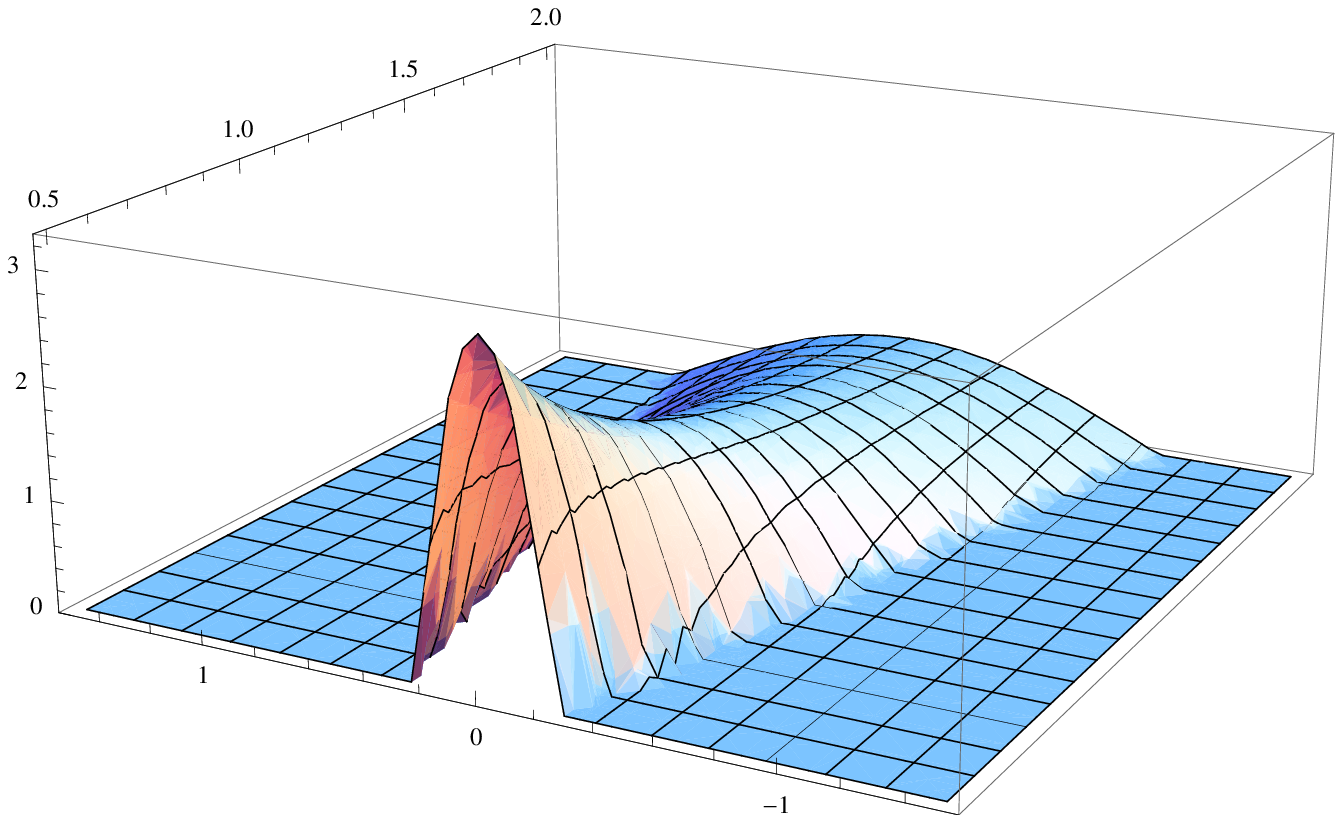}
\captionsetup{labelformat=empty} \caption{Figure 3: \
Continuous binomial density $\ d_x(s)\ $ for $\ x \in [0.7,2] \ $ and $\ s \in [-1.5,1.5].$}
 \label{f}
\end{figure}

\begin{proof}To show  property 1 it is enough to check that $\ d_x \ $ is an even function:
$$d_x(-s) \ = \ \frac{1}{b(x)}{x \brace  \frac{x}{2}-s} \ = \ \frac{1}{b(x)}{x \brace x - (\frac{x}{2}+s)}
 \ = \ \frac{1}{b(x)}{x \brace  \frac{x}{2}+s} \ = \ d_x(s),$$
where we have used Corollary \ref{by}. \ \ The moment of order $\ 2k \ $ is given
$$E(s^{2k}) \ = \ \frac{1}{b(x)}\int_{-\frac{x}{2}}^{\frac{x}{2}}s^{2k}{x \brace  \frac{x}{2}+s}ds.$$
Making the change of variable $\ t = \frac{x}{2} + s \ $ we obtain that:
$$E(s^{2k}) \ \ = \ \ \frac{1}{b(x)}\sum_{l=0}^{2k}\binom{2k}{l}(-\frac{x}{2})^{2k-l}\int_{0}^{x}t^{l}{x \brace  t}dt.$$
Property 2 follows from this identity using Theorem \ref{zz}.

Consider Property 3. We showed in \cite{cd} that the function $\ {x \brace s} \ = \ \mathrm{vol}(\Gamma(s, x-s)),\ $  for $0 \ \leq s \ \leq x,$ achieves its maximum at $\ s= \frac{x}{2}. \ $ Thus, $ \ d_x \  $ achieves its maximum at $\ s =0. \ $

Property 4 follows since the functions $\ d_x \ $ are non-negative, almost continuous (with discontinuity points  $\ -\frac{x}{2} \ $ and
$\ \frac{x}{2}, \ $) of total mass $\ 1, \ $ and support in the interval $\ [-\frac{x}{2}, \frac{x}{2}]. \ $ Let $\ f(s) \ $ be a continuous function  on $\ \mathbb{R}. \ $ Given $\ \epsilon > 0 \ $ choose $\ x  \ $ small enough such that $\ |f(s) - f(0)| < \epsilon \ $ for all $\ s \in [-\frac{x}{2}, \frac{x}{2}]. \ $ Under this conditions we have that $$ \bigg| \int_{-\infty}^{\infty} f(s)d_x(s)ds \ - \ f(0) \bigg|  \ \ \leq \ \ \int_{-\frac{x}{2}}^{\frac{x}{2}} \big|f(s)-f(0)\big| d_x(s)ds \ \ < \ \epsilon.$$

\end{proof}

\begin{rem}{\em The continuous binomial density $d_x(s)$  is plotted in Figures 3 \ and \ 4.  The reader should note the remarkable similarity
with the plots for the Brownian motion density, a subject that deserves further study.  }
\end{rem}

Although Theorem \ref{cn} already shows the combinatorial nature of the continuous  binomial
coefficients $\ {x \brace s}, \ $ the combinatorial interpretation is somewhat obscure due to the presence of negative signs.
This problem, as shown below, can be easily overcome by performing the change of variables
$\  (s,x)   \longrightarrow   (s, (t+1)s). \ $

\begin{thm}\label{pbc}
{\em For $\ t \geq 0\ $ and $\ s \geq 0\ $ we have that:
$$ \ {(t+1)s \brace s} \ = \  2+s+ \sum_{n=1}^\infty\frac{t^n}{n!}\Big(2(2n)_n\frac{s^{2n}}{(2n)!} +
(2n+1)_{n}\frac{s^{2n+1}}{(2n+1)!} + (2n-1)_{n}\frac{s^{2n-1}}{(2n-1)!}\Big) .$$
}
\end{thm}

\begin{proof}
$${(t+1)s \brace s}  \  =  \ \mathrm{vol} (\Gamma(s,ts))   \ =  \
\sum_{n=0}^\infty\Big(2\frac{s^{n}s^nt^n}{n!n!}\ + \ \frac{s^{n+1}s^{n}t^n}{(n+1)!n!}\ +\ \frac{s^{n}s^{n+1}t^{n+1}}{n!(n+1)!}\Big)\ =  $$
$$\sum_{n=0}^\infty\Big(2(2n)_n\frac{s^{2n}t^n}{(2n)!n!}\ + \ (2n+1)_{n}\frac{s^{2n+1}t^n}{(2n+1)!n!}\ +\ (2n+1)_{n+1}\frac{s^{2n+1}t^{n+1}}{(2n+1)!(n+1)!}\Big)\ = $$
$$2+s+ \sum_{n=1}^\infty\Big(2(2n)_n\frac{s^{2n}t^n}{(2n)!n!}\ + \ (2n+1)_{n}\frac{s^{2n+1}t^n}{(2n+1)!n!}\ +\ (2n-1)_{n}\frac{s^{2n-1}t^{n}}{(2n-1)!n!}\Big).  $$
\end{proof}

Theorem \ref{pbc} can be understood in terms of combinatorial species \cite{berg, Blan2, ecd, dp, j2} as follows.
Let $ \ \mathbb{B} \ $ be the category of finite sets and bijections, and $\ \mathrm{set} \ $ be the category of finite sets and maps.
Let $\ \mathrm{Inj}:\mathbb{B}\times \mathbb{B} \ \longrightarrow \mathrm{set}\ $ be the functor sending a pair
of finite sets $\ (a,b)\ $ to the set $\ \mathrm{Inj}(a,b)\ $ of injective maps from $\ a \ $ to $\ b$.
\begin{center}
\psset{unit=0.5cm}
    \begin{pspicture}(-1,-3)(18,5)
        \pscircle[linewidth=1pt](0.5,2){3}
        \psdot[dotstyle=*,
    dotsize=4pt](-1,2.5)
    \rput(-1,3){$1$}
    \rput(1.5,3.3){$2$}
    \psdot[dotstyle=*,
    dotsize=4pt](1.5,2.8)
    \psdot[dotstyle=*,
    dotsize=4pt](0.4,1.5)
    \rput(0.4,2){$3$}
    \rput(1.4,0){$4$}
    \psdot[dotstyle=*,
    dotsize=4pt](1.4,0.5)
    \pscircle[linewidth=1pt](14,2){3}
    \psdot[dotstyle=*,
    dotsize=4pt](12.5,3)
    \rput(12.5,3.5){$1$}
    \psdot[dotstyle=*,
    dotsize=4pt](14,3)
    \rput(14,3.5){$2$}
    \psdot[dotstyle=*,
    dotsize=4pt](15.5,3)
    \rput(15.5,3.5){$3$}
    \psdot[dotstyle=*,
    dotsize=4pt](13.2,1.8)
    \rput(13.2,2.3){$4$}
    \psdot[dotstyle=*,
    dotsize=4pt](14.7,1.8)
    \rput(14.7,2.3){$5$}
     \psdot[dotstyle=*,
    dotsize=4pt](12.5,0.3)
    \rput(12.5,0.8){$6$}
    \psdot[dotstyle=*,
    dotsize=4pt](14,0.3)
    \rput(14,0.8){$7$}
    \pscurve[linewidth=2pt]{->}(-0.8,2.8)(1.5,4)(8,4.7)(14,4)(15.2,3.5)
    \pscurve[linewidth=2pt]{->}(1.8,3)(7.5,3.8)(12.2,3.2)
    \pscurve[linewidth=2pt]{->}(0.6,1.7)(7,1)(13,2)
    \pscurve[linewidth=2pt]{->}(1.6,0.5)(7,-1)(13.9,0.2)
     \rput(8,-2.4){Figure 4. A map in  $\ B([4], [7])\ $ according to condition 5.}
    \end{pspicture}
    \end{center}
Consider the functor $\ B:\mathbb{B}\times \mathbb{B} \ \longrightarrow \mathrm{set}\ $ defined on $\ (a,b)  \in \mathbb{B} \times \mathbb{B}\ $ by
$$1) \ \ [2] \ \    \mbox{if} \ \  a=b=\emptyset;  \  \ \ \ \  2) \ \ [1] \ \    \mbox{if} \ \ a=\emptyset, \ |b|=1; \ \ \ \ \  3) \ \ [2]\times \mathrm{Inj}(a,b) \ \    \mbox{if} \ \   |b|=2|a| ;$$
$$4) \ \ \mathrm{Inj}(a ,b) \ \    \mbox{if} \ \ \ |b|=2|a|+1; \ \ \ \ \ 5) \ \ \mathrm{Inj}(a,b) \ \    \mbox{if} \ \ |b|=2|a|-1;
\ \ \ \ \  6) \ \ \emptyset \ \   \mbox{otherwise}.$$

Figure 4 \ shows a map  contributing to $\ B \ $ trough condition 5 above.

\begin{cor}{\em The generating function of $\ B \ \ $ is
$\ \ \displaystyle {(t+1)s \brace s}. $
}
\end{cor}

\begin{proof}
The result follows from the definition of $B$,  Theorem \ref{pbc}, the definition of the generating function
$$ \sum_{n,m=0}^{\infty}\big|B([n],[m])\big| \frac{t^ns^m}{n!m!},$$ and the fact that $\ (a)_n \ $ counts injective
functions from $\ [n] \ $ to $\ [a].$
\end{proof}

\begin{prop}{\em
The following  identity holds
$$ \ {2s \brace s} \ =  \ 2\sum_{n=0}^\infty {n \choose \lfloor n/2 \rfloor }\frac{s^{n}}{n!}\ = \ 2\big(I_0(2s)  +
 I_1(2s)\big).$$ }
\end{prop}

Thus, quite pleasantly,  the midpoint continuous binomial $\ {2s \brace s}\ $ is twice the generating function of the midpoint binomial coefficients.

To obtain a continuous analogue for the binomial coefficients we used their combinatorial interpretation as paths in a suitable lattice, and thus our interpretation for the continuous binomial coefficients counts directed paths in the corresponding direct manifold. The usefulness of the interpretation of the binomials coefficients as counting subsets of a fixed cardinality can hardly be overstated, so it is natural to ponder whether  an analogue interpretation is available for the continuous binomial coefficients.\\

Let $\ \mathrm{U}[x,s] \ $ be the family of subsets $\ S \ \subseteq \ [0,x]\ $ such that:
\begin{itemize}
\item $S$ is a finite disjoint union of closed subintervals of $\ [0,x]$.
\item The sum of the lengths of the closed subintervals defining $\ S\ $ is equal to $\ s$.
\end{itemize}

The linear order on $\ [0,x] \ $ induces a linear order on the closed subintervals defining a set $\ S \in \mathrm{U}[x,s].\ $ Consider the  map $$ \mathrm{path}: \mathrm{U}[x,s]\ \longrightarrow \ \Gamma(s,x-s) $$
sending  $\ S \ = \ [a_1,b_1]\ \sqcup \ [a_2,b_2] \ \sqcup \ \cdots \ \sqcup \ [a_n,b_n] \ $ in $ \ \mathrm{U}[x,s],\ $ written in the linear order, to the directed path in $\ \Gamma(s,x-s) \ $ constructed as follows (see Figure 5 where $\ S \ $ is the union of the marked subintervals on the left):
\begin{itemize}
 \item For $\ S = \emptyset, \ $ a valid choice if and only if $ \ s=0, \ $ the associated path has format $\ (2) \ $ and time distribution $\ (x).$
  \item If $\ a_1=0\ $ and $\ b_n=x, \ $ then the associated path has format $\ (1,2,...,2,1)\ $ of length $\ 2n-1 \ $ and time distribution
  $\ (b_1,a_2-b_1,b_2-a_2,...,x-a_n).$
  \item  If $\ a_1=0\ $ and $\ b_n<x, \ $ then the associated path has format $\ (1,2,...,1,2)\ $ of length $\ 2n \ $ and time distribution
  $\ (b_1,a_2-b_1,b_2-a_2,...,b_n -a_n, x-b_n).$
   \item If $\ a_1>0 \ $ and $\ b_n=x, \ $ then the associated path has format $\ (2,1,...,2,1)\ $ of length $\ 2n \ $ and time distribution
  $\ (a_1,b_1-a_1,a_2-b_1,...,x-a_n).$
  \item  If $\ a_1>0 \ $ and $\ b_n<x, \ $ then the associated path has format $\ (2,1,...,1,2) \ $ of length $\ 2n+1 \ $ and time distribution
  $\ (a_1,b_1-a_1,a_2-b_1,...,b_n -a_n, x-b_n).$
\end{itemize}

\begin{center}
\psset{unit=1cm}
    \begin{pspicture}(-2,-1)(14,3)
        \psline[linewidth=1pt]{->}(0,0)(6,0)
        \rput(0,-0.4){$0$}
        \psline[linewidth=3pt]{}(0,0)(1.8,0)
        \rput(1.8,-0.4){$1.8$}
        \psline[linewidth=3pt]{}(2.5,0)(3.5,0)
        \rput(2.5,-0.4){$2.5$}
        \rput(3.5,-0.4){$3.5$}
        \psline[linewidth=3pt]{}(4.2,0)(5,0)
        \rput(4.2,-0.4){$4.2$}
        \rput(5,-0.4){$5$}
        \psline[linewidth=1pt]{->}(8,0)(8,3)
        \psline[linewidth=1pt]{->}(7,0)(13.5,0)
        \psline[linewidth=2pt]{->}(8,0)(9.8,0)
        \psline[linewidth=2pt]{->}(9.8,0)(9.8,0.7)
        \psline[linewidth=2pt]{->}(9.8,0.7)(10.8,0.7)
        \psline[linewidth=2pt]{->}(10.8,0.7)(10.8,1.4)
        \psline[linewidth=2pt]{->}(10.8,1.4)(13,1.4)
        \rput(9.8,-0.4){$1.8$}
        \rput(7.6,0.7){$0.7$}
        \rput(10.8,-0.4){$2.8$}
        \rput(7.6,1.4){$1.4$}
        \rput(13,-0.4){$3.6$}
        \rput(7,-1){Figure 5. Set  $\ S \ $ in $\  \mathrm{U}[0,5] \ $ and its associated directed path.}
    \end{pspicture}
    \end{center}
It is easy to see that the map $\ \mathrm{path} \ $ is  injective. Moreover the map  $\ \mathrm{path} \ $ is essentially surjective, i.e. the image of $\ \mathrm{path} \ $ has full measure. Indeed, to show that $$\ \displaystyle \mathrm{vol}\big(  \mathrm{\Gamma}(s,x-s) \setminus \mathrm{path}(\mathrm{U}[x,s] )\big) \ = \ 0 $$ one simply notes
that a path in $\ \displaystyle \mathrm{\Gamma}(s,x-s) \setminus \mathrm{path}(\mathrm{U}[x,s] ) \ $ must have at least one coordinate equal to zero, and therefore the later set is included in a finite union of codimension one subsets, which implies that it has to be a set of measure zero.

Since $\ \mathrm{path} \ $ is a bijection onto its image, a set of full measure, one obtains by pull-back  a measure-procedure  on $\ \mathrm{U}[x,s]. \ $ Thus we have shown the following result.

\begin{prop}{\em For $0 \leq s \leq x, \ $  we have that:
$${x \brace s} \ = \ \mathrm{vol}\big( \mathrm{U}[x,s]\big) \ = \ \sum_{n=0}^{\infty}\mathrm{vol}\big( \mathrm{U}_n[x,s]\big),$$
where $\ \mathrm{U}_n[x,s] \ \subseteq \ \mathrm{U}[x,s] \ $ consists of sets which are the union of $\ n \ $ closed subintervals.
}
\end{prop}

\section{Continuous Catalan Numbers}\label{5}

We proceed to construct continuous analogues for the Catalan numbers  $\  c_n  =  \frac{1}{n+1}{2n
\choose n}.\  $ The Catalan numbers admit a myriad of interesting combinatorial interpretations, see Stanley's book \cite{rs}. Among those we work with a lattice path interpretation because that is what is needed for our present purposes. \\

Consider  step vectors $\ V = \{(1,1), (1,-1)\} \ \subseteq \
\mathbb{Z}^2 \ \subseteq \ \mathbb{R}^2. \ $ It is well-known  that the Catalan numbers count Dyck paths (see Figure 5), i.e.:
$$c_n \ =  \  \big|\{ \mbox{lattice \ paths  \ in } \ \mathbb{R}_{\geq 0}^2 \ \ \mbox{from} \ \ (0,0) \ \ \mbox{to} \ \ (2n,0)\}\big|.$$
\begin{center}
\psset{unit=0.5cm}
    \begin{pspicture}(-1,-3)(18,5)
        \psgrid[subgriddiv=1,griddots=8,gridlabels=8pt](-2,-1)(18,5)
        \psline[linewidth=1pt]{->}(0,0)(3,3)
        \psline[linewidth=1pt]{->}(3,3)(5,1)
        \psline[linewidth=1pt]{->}(3,3)(5,1)
        \psline[linewidth=1pt]{->}(5,1)(6,2)
        \psline[linewidth=1pt]{->}(6,2)(7,1)
        \psline[linewidth=1pt]{->}(0,0)(0,5)
        \psline[linewidth=1pt]{->}(-2,0)(18,0)
        \psline[linewidth=1pt]{->}(7,1)(10,4)
        \psline[linewidth=1pt]{->}(10,4)(12,2)
        \psline[linewidth=1pt]{->}(12,2)(13,3)
        \psline[linewidth=1pt]{->}(13,3)(16,0)
        \rput(8,-2.4){Figure 5. A Dyck lattice path from $\ 0 \ $ to $\ 16$.}
    \end{pspicture}
    \end{center}
If  such a path has a pattern of length $\ 2k, \ $ then it has $\ k \ $ peaks and $\ k-1 \ $ valley points. Therefore pattern decomposition induces the counting of Dyck paths by the number of peaks \cite{nara0, nara1}, i.e. it leads to the Narayana identity
$$\frac{1}{n+1}{2n \choose n}\ = \ \sum_{k=1}^n \frac{1}{n}{n \choose k}{n \choose k-1}  .$$

To construct continuous analogues for the Catalan numbers we
consider the directed manifold $\ (\mathbb{R}^2,  (1,1),  (1,-1)) \ $
and  measure directed paths from $\ (0,0) \ $ to $\ (x,0) \ $
fully included in $\ \mathbb{R}_{\geq 0}^2, \ $ see Figure 4.
\begin{center}
\psset{unit=0.5cm}
    \begin{pspicture}(-2,-3)(16,5)
        \psgrid[subgriddiv=1,griddots=8,gridlabels=8pt](-2,-1)(16,5)
        \psline[linewidth=1pt]{->}(0,0)(2.75,2.75)
       \psline[linewidth=1pt]{->}(2.75,2.75)(4,1.5)
       \psline[linewidth=1pt]{->}(4,1.5)(5.25,2.75)
       \psline[linewidth=1pt]{->}(5.25,2.75)(6.75,1.25)
       \psline[linewidth=1pt]{->}(6.75,1.25)(8.85,3.35)
       \psline[linewidth=1pt]{->}(8.85,3.35)(10.9,1.3)
       \psline[linewidth=1pt]{->}(10.9,1.3)(12.2,2.6)
       \psline[linewidth=1pt]{->}(12.2,2.6)(14.8,0)
       \psline[linewidth=1pt]{->}(0,0)(0,5)
        \psline[linewidth=1pt]{->}(-2,0)(16,0)
    \rput(8,-2.4){Figure 6. A Dyck directed path from $\ 0 \ $ to $\ 14.8$.}
    \end{pspicture}
\end{center}
Given $\ (x,y) \in \mathbb{R}_{\geq 0}^2\ $ we let $\
\Lambda(x,y)\ $ be the moduli space of directed paths from $\
(0,0) \ $ to $\ (x,y) \ $ included in $\ \mathbb{R}_{\geq 0}^2 \ $
with patterns of the form $\ (1,2,\dots, 1,2) , \ $ and let
$\ \Lambda^n(x,y) \ $ be the set of directed paths with pattern of length $\ 2n+2. \ $ By construction $\ \mathrm{vol}(\Lambda^k(2n,0)) \ $ is a continuous analogue of the Narayana number $\ \frac{1}{n}{n \choose k}{n \choose k-1},  \  $  for $\ 0 \leq k < n, \ $ i.e. there are exactly  $\ \frac{1}{n}{n \choose k}{n \choose k-1}  \  $ integer points in the interior of $\ \Lambda^k(2n,0).$

\begin{defn}{\em For $\ 0 < y < x, \ $ in $ \ \mathbb{R} \ $ the two-variables continuous Catalan function $\ C(x,y) \ $ is given by
$$C(x,y) \ \ = \ \  \sum_{n=0}^{\infty} \mathrm{vol}(\Lambda^n(x,y)).$$
The domain of $\ C \ $ is extended to $\ 0 \leq y \leq x \ $ by continuity.
The one-variable continuous Catalan function is given
by $\ C(x) =  C(x,0) \ \ $ for $\ \ x \in \mathbb{R}_{\geq 0}.$
}
\end{defn}

\begin{prop}\label{w}
{\em Consider the moduli space  $\ \Lambda^n(x,y) \ $  for $\ 0 \leq y \leq x \ $ and $\ n \in \mathbb{N}.$
\begin{enumerate}
  \item  $ \Lambda^n(x,y) \ $  is the convex polytope given in simplicial coordinates by:
$$s_0 + \cdots + s_n \  = \ \frac{x+y}{2}, \ \ \ \ \ \ \ \  t_0 + \cdots + t_n \  = \ \frac{x-y}{2},$$
$$s_0 \  \geq \ t_0, \ \ \ \ s_0 + s_1 \  \geq \ t_0 + t_1, \ \ \ \ \cdots \ \ \ \  s_0 + \cdots + s_{n-1} \  \geq \ t_0 + \cdots + t_{n-1}.$$
In particular $ \ \Lambda^0(x,y) \ = \ \{ (\frac{x+y}{2}, \frac{x-y}{2}) \} \ $  and thus $\ \mathrm{vol}(\Lambda^0(x,y)) = 1.$
  \item For $n \geq 1, \ $ the convex polytope $ \ \Lambda^n(x,y) \ $ is given in Cartesian coordinates by:
$$0\leq x_1 \leq \cdots \leq x_n \leq \frac{x+y}{2}, \ \ \ \ \ 0\leq y_1 \leq \cdots \leq y_n \leq \frac{x-y}{2}, \ \ \ \ \ x_i \geq y_i.$$

\item For $n \geq 1, \ $ $\ \Lambda^n(x,y) \ $ is given in terms of valley points coordinates by:
$$0 \ \leq \ a_1 + b_1 \ \leq \ \cdots \ \leq \ a_n + b_n \ \leq \ x + y, $$
$$0 \ \leq \ a_1 - b_1 \ \leq \ \cdots \ \leq \ a_n - b_n \ \leq \ x - y. $$
\end{enumerate}
}
\end{prop}

\begin{proof}
Item 1 follows directly from Proposition \ref{p}. Item 2 follows from item 1 making the change of variables $\ \ x_i \ = \ s_0 \ + \ \cdots \ + t_{i-1}\ \ $
and $\ \ y_i \ = \ t_0 \ + \ \cdots \ + t_{i-1}\ \ $ for
$\ 1 \leq i \leq n.\ $  Item 3 follows from item 2 making the change of variables $\ a_i=x_i + y _i \ $ and $\ b_i = x_i - y_i.$
\end{proof}

\begin{cor}{\em
The infinite sum defining $ \ C(x,y) \ $ is convergent and uniformly convergent on bounded sets.}
\end{cor}

\begin{proof}
From Proposition \ref{w} we have that
$$0 \ \leq \ \sum_{n=0}^{\infty} \mathrm{vol}(\Lambda^n(x,y))\  \leq \   \sum_{n=0}^{\infty} \mathrm{vol}(\Delta_{n}^{\frac{x+y}{2}})
 \mathrm{vol}( \Delta_{n}^{\frac{x-y}{2}})
 \ = \  \sum_{n=0}^{\infty} \frac{(x+y)^n}{n!} \frac{(x-y)^n}{n!}.$$
The later series has the desired properties.
\end{proof}

\begin{exmp}{\em For $\ 0 \leq y \leq x,\ $ the polytope  $ \ \Lambda^{1}(x,y)\ $ is given by
$$ 0 \leq x_1 \leq \frac{x+y}{2}, \ \ \ \ \  0 \leq y_1 \leq \frac{x-y}{2}, \ \ \ \ \ y_1 \leq x_1.$$
Applying the change of variables $\ (a,b)=(x_1 - y_1, y_1) \ $ (with Jacobian determinant $1$)  we obtain the polytope given by
$$ 0 \ \leq b \ \leq \ \frac{x-y}{2}, \ \ \ \ \ \ \  0 \ \leq a \ \leq \ \frac{x+y}{2}-b.$$
Thus we have that: $$\mathrm{vol}(\Lambda^{1}(x,y)) \  =  \ \int_{0}^{\frac{x-y}{2}} \int_{0}^{\frac{x+y}{2}-b}1dadb \ = \
\frac{1}{8}(x-y)(x+3y) \ \ \ \ \mbox{and}
\ \ \ \ \mathrm{vol}(\Lambda^{1}(x,0)) \   =  \
\frac{1}{8}x^2.$$}
\end{exmp}

\begin{prop}\label{z}
{\em For $\ 0  \leq y  \leq x, \ $ and $ \ n  \in \mathbb{N}\  $ the following recursion holds:
$$\mathrm{vol}(\Lambda^{n+1}(x,y))  \ \ = \ \ \int_0^{\frac{x-y}{2}} \int_0^{\frac{x+y}{2} -b}\mathrm{vol}(\Lambda^{n}(a+2b,a)) dadb. $$
}
\end{prop}

\begin{proof}
Consider Cartesian coordinates on $\Gamma^{n+1}(x,y):$
$$ 0\leq x_1 \leq \cdots \leq x_n \leq x_{n+1} \leq \frac{x+y}{2}, \ \ \ \ \ 0 \leq y_1 \leq \cdots \leq y_n \leq y_{n+1} \leq \frac{x-y}{2}, \ \ \ \ \ x_i \geq y_i.$$ Making the change of variables $\ (a,b) = (x_{n+1}-y_{n+1}, y_{n+1}) \ $ with Jacobian $\ 1, \ $ we obtain the polytope given by:
$$0 \leq b \leq \frac{x-y}{2}, \ \ \ \ \  0 \leq a \leq \frac{x+y}{2} -b,$$
$$ 0\leq x_1 \leq \cdots \leq x_n \leq a+b, \ \ \ \ \ 0 \leq y_1 \leq \cdots \leq y_n \leq b, \ \ \ \ \ x_i \geq y_i.$$
Therefore we have that:
$$\mathrm{vol}(\Lambda^{n+1}(x,y))  \ \ = \ \ \int_{\Lambda^{n+1}(x,y)}1dx_1\cdots dx_{n+1}dy_1\cdots dy_{n+1} \ \ = $$
$$\int_0^{\frac{x-y}{2}} \int_0^{\frac{x+y}{2} -b} \int_{\Lambda^{n}(a+2b,a)}1dx_1\cdots dx_{n}dy_1\cdots dy_{n} dadb \ \ = $$
$$ \int_0^{\frac{x-y}{2}} \int_0^{\frac{x+y}{2} -b}\mathrm{vol}(\Lambda^{n}(a+2b,a)) dadb.$$

\end{proof}

\begin{cor}\label{ul}
{\em For $\ 0  \leq y  \leq  x, \ $ and $\ n \in \mathbb{N}_{>0} \ $ the function $\ \mathrm{vol}(\Lambda^{n}(x,y)) \ $ is given by:
$$ \int_0^{\frac{x-y}{2}} \int_0^{\frac{x+y}{2} -b_1} \int_0^{b_1} \int_0^{a_1+b_1 -b_2}\cdots
 \int_0^{b_{n-1}} \int_0^{a_{n-1}+b_{n-1} -b_{n}} 1  da_ndb_nda_{n-1}db_{n-1} \cdots  da_1db_1. $$
}
\end{cor}

\begin{proof}
Follows iterating Propositon \ref{z}.
\end{proof}

\begin{prop}{\em The Catalan function $\  C(x,y) \ $ satisfies the integral equation:
$$C(x,y) \ = \  1 \ + \ \int_0^{\frac{x-y}{2}} \int_0^{\frac{x+y}{2} -b} C(a+2b,a) dadb. $$
}
\end{prop}

\begin{proof}
Follows from Propositions  \ref{w} and \ref{z}.
\end{proof}

For $\ n \geq 0 \ $ we define by recursion the  functions $\ I_n: \mathbb{R}_{\geq 0}^2 \longrightarrow \mathbb{R}\ $ as follows:
$$I_0 \ = \ 1, \ \ \ \ \ \mbox{and} \ \ \ \ \ I_n(a,b) \ = \ \int_{0}^{b}\int_{0}^{a+b-v}I_{n-1}dudv. $$
The function $\ I_n \ $ is polynomial and admits a finite expansion:
$$I_n(a,b) \ = \ \sum_{k,l}I^n_{k,l}\frac{a^k}{k!}\frac{b^l}{l!}  \ \ \ \ \ \mbox{with} \ \ \ I^n_{k,l} \in \mathbb{Z}.$$

\begin{lem}{\em For $l \geq 1,$ and $n \geq 1,$ the coefficient $ \ I^n_{k,l} \ $ satisfies the recuersion:
$$ I^n_{k,l} \ = \ \sum_{p=0}^{l-1}\sum_{q=0}^{l-p-1} (-1)^q {l \choose p}{l-p-1 \choose q}I^{n-1}_{k+p+q-1,\ l-p-q-1}.$$
Moreover: $\ I^0_{k,l} \ = \ \delta_{k0}\delta_{l0}, \ \ \ I^n_{k,l} \ = \ 0 \ \ $ for $\ \ n > l, \ \ \ $ and
$\ \ I^n_{k,0} \ = \ 0 \ $ for $\ n > 0.$
}
\end{lem}

\begin{proof} We have that
$$I_n(a,b) \ = \ \sum_{k,l}I^{n-1}_{k,l}\int_{0}^{b}\int_{0}^{a+b-v}\frac{u^k}{k!}\frac{v^l}{l!}dudv \ = $$
$$\sum_{k,l}I^{n-1}_{k,l}\int_{0}^{b}\frac{(a+b-v)^{k+1}}{(k+1)!}\frac{v^l}{l!}dv \ \ = \ \
\sum_{k,l, \ r+p+q=k+1}(-1)^qI^{n-1}_{k,l}\frac{a^r}{r!}\frac{b^p}{p!}\int_{0}^{b}\frac{v^{q}}{q!}\frac{v^l}{l!}dv \ = $$
$$\sum_{k,l,\ r+p+q=k+1}(-1)^q{p+q+l+1 \choose p}{q+l \choose q}I^{n-1}_{k,l}\frac{a^r}{r!}\frac{b^{p+q+l+1}}{(p+q+l+1)!} \ = $$
$$\sum_{r,\ s>0}\bigg(\sum_{p=0}^{s-1}\sum_{q=0}^{s-p-1}(-1)^q{s \choose p}{s-p-1 \choose q}I^{n-1}_{r+p+q-1, \ s-p-q-1}  \bigg) \frac{a^r}{r!}\frac{b^s}{s!}, $$
where we have set $\ s=p+q+l+1 \ $ to get the last expression.

\end{proof}

Our final result provides an explicit formula for the Catalan function $ \ C(2x). \ $

\begin{thm}{\em  For $\ x \geq 0, \ $ the Catalan function $ \ C(2x) \ $ is given by
$$ C(2x) \ \ = \ \  1 \ + \  \sum_{m=2}^{\infty}\bigg(\sum_{k+l=m-2}(\sum_{n=1}^{l+1}I^{n-1}_{k,l}) \sum_{p=0}^{k+1}(-1)^{k+1-p}{m \choose p}{m-p-1 \choose l}\bigg) \frac{x^m}{m!}.$$

}
\end{thm}

\begin{proof}
Using Corollary \ref{ul}  we get that
$$C(2x) \ = \ C(2x,0) \ = \ 1 \ + \ \sum_{n=1}^{\infty}\int_{0}^{x}\int_{0}^{x-b}I_{n-1}(a,b)dadb\ \ = $$
$$1\ + \ \sum_{n=1}^{\infty}  \sum_{k,l}I^{n-1}_{k,l} \int_{0}^{x}\int_{0}^{x-b}\frac{a^k}{k!}\frac{b^l}{l!}dadb \ \ = $$
$$1\ + \ \sum_{n=1}^{\infty}  \sum_{k,l}I^{n-1}_{k,l}\sum_{p=0}^{k+1}(-1)^{k+1-p}{k+l+2 \choose p}{k+l+1-p \choose l} \frac{x^{k+l+2}}{(k+l+2)!} \ \ = $$
$$1 \ + \  \sum_{m=2}^{\infty}\bigg(\sum_{k+l=m-2}(\sum_{n=1}^{l+1}I^{n-1}_{k,l}) \sum_{p=0}^{k+1}(-1)^{k+1-p}{m \choose p}{m-p-1 \choose l}\bigg) \frac{x^m}{m!} .$$
\end{proof}

\noindent $\mathbf{Acknowledgements.}$ \ \ We thank Tom Koornwinder for kindly pointing out to us the connection between the continuous binomial coefficients and the Bessel functions, his comments and suggestions lead to substantial improvements on a early version of this work. We also thank  Jos\'e Luis Ram\'irez.

\noindent lnrdcano@gmail.com \\
\noindent Departamento de Matem\'aticas, Universidad Sergio Arboleda, Bogot\'a, Colombia\\

\noindent ragadiaz@gmail.com\\
\noindent Departamento de Matem\'aticas, Pontificia Universidad Javeriana, Bogot\'a, Colombia

\end{document}